\documentclass[10pt]{article}%
\usepackage{amsmath}
\usepackage{amsfonts}
\usepackage{amssymb}
\usepackage{graphicx}%
\usepackage{mathrsfs}
\usepackage[square, comma, sort&compress, numbers]{natbib}
\usepackage{amssymb, color}
\usepackage[body={15.5cm,21cm}, top=3cm]{geometry}%
\usepackage{hyperref}
\providecommand{\U}[1]{\protect \rule{.1in}{.1in}}

\DeclareMathOperator*{\esssup}{ess\,sup}
\newtheorem{theorem}{Theorem}[section]

\newtheorem{definition}[theorem]{{Definition}}

\newtheorem{lemma}[theorem]{Lemma}

\newtheorem{remark}[theorem]{{Remark}}

\newenvironment{proof}[1][Proof]{\noindent \textbf{#1.} }{\  \rule{0.5em}{0.5em}}

\begin{document}

\title{General Mean Reflected BSDEs}
\author{ Ying Hu\thanks{Univ. Rennes, CNRS, IRMAR-UMR 6625, F-35000, Rennes, France. 
		ying.hu@univ-rennes1.fr.
		Research   supported
		by the Lebesgue Center of Mathematics ``Investissements d'avenir"
		program-ANR-11-LABX-0020-01
		.}
	\and Remi Moreau \thanks{Univ. Rennes, CNRS, IRMAR-UMR 6625, F-35000, Rennes, France. remi.moreau@ens-rennes.fr. Research   supported
		by the Lebesgue Center of Mathematics ``Investissements d'avenir"
		program-ANR-11-LABX-0020-01. }
	\and Falei Wang\thanks{Zhongtai Securities Institute for Financial  Studies and School of Mathematics, Shandong University, Jinan 250100, China.
		flwang2011@gmail.com. Research supported by the Natural Science Foundation of Shandong Province for Excellent Youth Scholars (ZR2021YQ01),  the National Natural Science Foundation of China (Nos. 12171280,  12031009, and 11871458) and the Young Scholars Program of Shandong University.}}
\date{}
\maketitle
\begin{abstract}
The present paper is devoted to the study of backward stochastic differential equations with mean reflection formulated by Briand et al. \cite{BH}. We investigate the solvability of a generalized mean reflected BSDE, whose driver also depends on the distribution of the solution term $Y$. Using a fixed-point argument, BMO martingale theory and  the $\theta$-method, we establish the existence and uniqueness result for such BSDEs in several  typical situations, including the case where the driver is quadratic with bounded or unbounded terminal  condition.
\end{abstract}

\medskip
\noindent \textbf{MSC-classification}: 60H10, 60H30

\section{Introduction}

Let $(\Omega, \mathscr{F},\mathbf{P})$ be  a given complete probability space  under which $B$ is a $d$-dimensional standard Brownian motion. Suppose $(\mathscr{F}_t)_{0\leq t\leq T}$ is the corresponding natural filtration augmented by the $\mathbf{P}$-null sets and $\mathcal{P}$ is  the  sigma algebra of progressive sets of $\Omega\times [0,T]$. In this paper, we consider the following backward stochastic differential equation (BSDE) with mean reflection:
\begin{align}\label{my1}
\begin{cases}
&Y_t=\xi+\int_t^T f(s,Y_s,\mathbf{P}_{Y_s},Z_s)ds-\int_t^T Z_sdB_s+K_T-K_t, \quad 0\le t\le T,\\
& \mathbf{E}[\ell(t,Y_t)]\geq 0,\quad \forall t\in [0,T] \,\,\mbox{ and }\,\, \int_0^T\mathbf{E}[\ell(t,Y_t)]dK_t = 0,
\end{cases}
\end{align}
where $\mathbf{P}_{Y_t}$ is the marginal probability distribution of the process $Y$ at time $t$,  the terminal condition $\xi$ is a scalar-valued $\mathscr{F}_T$-measurable random variable, the driver $f: \Omega\times [0,T]\times \mathbb{R}\times \mathcal{P}_1(\mathbb{R})\times \mathbb{R}^d\to \mathbb{R}$, and  the running loss function $\ell:\Omega\times[0,T]\times\mathbb{R} \to \mathbb{R}$ are  measurable maps with respect to $\mathcal{P}\times \mathcal{B}(\mathbb{R}) \times \mathcal{B}(\mathcal{P}_1(\mathbb{R})) \times  \mathcal{B}(\mathbb{R}^{d})$ and $\mathscr{F}_T\times\mathcal{B}([0,T])\times\mathcal{B}(\mathbb{R})$ respectively.  Our aim is to prove that the mean reflected BSDE \eqref{my1} admits a unique deterministic solution $(Y,Z,K)$, in the sense that $K$ is a deterministic, non-decreasing, and continuous process starting from the origin.

\medskip

BSDEs  with mean reflection were first introduced by Briand et al. in \cite{BH} to deal with the super-hedging problem under running risk management constraints. When the driver $f$ is independent of the probability distribution of $Y_t$, the authors of \cite{BH} established the existence and uniqueness of the deterministic solution $K$ to the mean reflected BSDE \eqref{my1} based on the following representation 
\begin{align}\label{my10242}
K_t=\sup_{0\leq s\leq T} L_s\bigg(\mathbf{E}_s\bigg[\xi+\int_t^T f(r,Y_r,Z_r)dr\bigg]\bigg)- \sup_{t\leq s\leq T} L_s\bigg(\mathbf{E}_s\bigg[\xi+\int_t^T f(r,Y_r,Z_r)dr\bigg]\bigg),
\ \forall t\leq T,\end{align}
where the map $t\mapsto L_t$ is given by \eqref{my10241} for $t\in[0,T]$. With the help of this  representation result, they were able to construct a contraction mapping when $f$ is uniformly Lipschitz continuous in both variables $Y,Z$. 
For more details on this topic, we refer the reader to \cite{BC1, BC, DES, HMW} and the references therein.

\medskip

In particular, combining BMO martingale theory and a fixed-point method, Hibon et al. \cite{HH1} extended the results from \cite{BH} to the case with bounded terminal condition, when the driver $f$ is allowed to have quadratic growth in the second unknown $z$.
However, in order to estimate the solution $K$ with the representation \eqref{my10242}, they need to assume the following additional condition on the driver:
\[
\text{ $(t,y)\mapsto f(t,y,0)$ is uniformly  bounded,}
\]
which is not necessary for the solvability  of quadratic BSDEs with bounded terminal conditions (see \cite{BE1, K1, zhang2017}).

\medskip

One of our motivations  is to remove this additional
assumption. The key point of our fixed-point method is based on the following representation result:
\begin{align}\label{my10243}
K_t=\sup_{0\leq s\leq T} L_s(y_s)- \sup_{t\leq s\leq T} L_s(y_s), \ \forall t\leq T,
\end{align}
where $y$ denotes the solution to the following BSDE
 \begin{align*}
y_t=\xi+\int^{T}_t f(s,Y_s,\mathbf{P}_{Y_s},z_s)ds-\int^{T}_t z_s dB_s.
\end{align*}
Compared with \eqref{my10242}, the representation result  \eqref{my10243} for the deterministic solution $K$  does not explicitly involve the term $Z$. We can then make use of relevant BSDE techniques to estimate  the solution $K$ and establish existence and uniqueness of the solution to the quadratic mean reflected BSDE \eqref{my1} without this additional assumption.
\medskip

Moreover, this method can also be used to solve BSDEs with mean reflection under weak assumptions on the data. Indeed, with the help of the corresponding BSDE theory and the representation result  \eqref{my10243}, we make a counterpart study
for the case where the driver is Lipschitz and the terminal condition admits a $p$th-order moment. We also tackle the situation with quadratic driver and unbounded terminal condition. 
In the first case, we apply the representation result \eqref{my10243} and  a linearization technique to derive  a priori estimates and build a contraction mapping. 

\medskip

Note that the comparison theorem does not hold for mean reflected BSDEs (see \cite{HH1}). Thus the monotone convergence argument is quite restrictive for quadratic BSDEs with mean reflection, which  differs from the  quadratic BSDEs case, see, e.g., \cite{BE0,BH2006,K1}. Borrowing some ideas from \cite{BH2008, FHT2}, we use the representation result \eqref{my10243} and a $\theta$-method to give a successive approximation procedure when the driver is quadratic and the terminal condition admits exponential moments of arbitrary order. 
\medskip

The main contribution herein is that we introduce a new representation result  to develop the mean reflected BSDEs theory.   In particular, we establish the
well-posedness of equation \eqref{my1} with mean reflection for several typical situations.  Compared to \cite{HH1}, our argument also removes the additional condition in the quadratic case with bounded terminal condition.

\medskip

The paper is organized as follows. In section 2, we start with  the Lipschitz case to illustrate the main idea. Section 3 is devoted to the study of the quadratic case with bounded terminal condition. We remove the additional boundedness condition in Section 4 assuming convexity on the driver.

\subsubsection*{Notations.}
For  each Euclidian space,  we  denote by $\langle\cdot ,\cdot \rangle$  and  $|\cdot|$ its scalar product and the associated norm, respectively.
Then, for each $p\geq 1$, we consider the following collections:

\begin{description}
	\item[$\bullet$] $ \mathcal{L}^{p}$ is the collection of  real-valued $\mathscr{F}_T$-measurable random variables $\xi$ satisfying {\small \[
	\|\xi\|_{\mathcal{L}^{p}}=\mathbf{E}\left[|\xi|^p\right]^{\frac{1}{p}}<\infty;
	\]}
	\item[$\bullet$]  $\mathcal{L}^{\infty}$ is the collection of  real-valued $\mathscr{F}_T$-measurable random variables $\xi$ satisfying 	{\small \[
 \|\xi\|_{\mathcal{L}^{\infty}}=\esssup\limits_{\omega\in\Omega}|\xi(\omega)|<\infty;\]}
	\item[$\bullet$]  $\mathcal{H}^{p,{d}}$  is the collection of $\mathbb{R}^{{d}}$-valued  $\mathscr{F}$-progressively measurable   processes $(z_t)_{0\leq t\leq T}$
	satisfying
	{\small \begin{align*}
	\|z\|_{\mathcal{H}^{p}}=\mathbf{E}\left[\bigg(\int^T_0|z_t|^2dt \bigg)^{\frac{p}{2}}\right]^{\frac{1}{p}}<\infty;
	\end{align*}}

	\item[$\bullet$]  $\mathcal{S}^p$  is the collection of real-valued  $\mathscr{F}$-adapted continuous  processes $(y_t)_{0\leq t\leq T}$ satisfying
{\small\begin{align*}
	\|y\|_{\mathcal{S}^{p}}=\mathbf{E}\bigg[\sup_{t\in[0,T]}|y_t|^p\bigg]^{\frac{1}{p}}<\infty;
	\end{align*}}
\item[$\bullet$]  $\mathcal{S}^{\infty}$  is the collection of  real-valued  $\mathscr{F}$-adapted continuous  processes $(y_t)_{0\leq t\leq T}$ satisfying
{\small \[
\|y\|_{\mathcal{S}^{\infty}}=\esssup\limits_{(t,\omega)\in[0,T]\times\Omega}|y(t,\omega)|<\infty;
\]}
\item[$\bullet$] $\mathcal{P}_p(\mathbb{R})$ is the collection of all probability measures over $(\mathbb{R},\mathcal{B}(\mathbb{R}))$ with finite $p^{th}$ moment, endowed with the $p$-Wasserstein distance $W_p$;
\item[$\bullet$] $ \mathbb{L}^{p}$ is the collection of  real-valued $\mathscr{F}_T$-measurable random variables $\xi$ satisfying $\mathbf{E}\left[e^{p|\xi|}\right]<\infty$;
	\item[$\bullet$]   $\mathbb{S}^p$ is the collection
of all stochastic processes $Y$ such that $e^Y\in \mathcal{S}^p$;
\item[$\bullet$] $\mathbb{L}$ is the collection of all random variables $\xi\in\mathbb{L}^p$ for any $p\geq 1$, and   $\mathcal{H}^d$ and $\mathbb{S}$ are defined similarly;
\item[$\bullet$]  $\mathcal{A}$ is the collection of deterministic, non-decreasing, and continuous processes $(K_t)_{0\leq t\leq T}$ starting from the origin, i.e. $K_0=0$;
\item[$\bullet$] $\mathcal{T}_{t}$ is the collection of $[0,T]$-valued $\mathcal{F}$-stopping times  $\tau$  such that $\tau \geq t$ $\mathbf{P}$-a.s.;
\item[$\bullet$]
 $BMO$  is the collection of    $\mathbb{R}^d$-valued progressively measurable  processes $(z_t)_{0\leq t\leq T}$  such that
\begin{align*}
\|z\|_{BMO}:=\sup\limits_{\tau\in\mathcal{T}_0}\esssup\limits_{\omega\in\Omega}\mathbf{E}_{\tau}\bigg[\int^T_{\tau}|z_s|^2ds\bigg]^{\frac{1}{2}} < \infty.
\end{align*}
\end{description}\smallskip
   We denote by $\ell_{[a,b]}$ the corresponding collections for the stochastic processes {with} time indexes on $[a,b]$ for $\ell=\mathcal{H}^{p,d},\mathcal{S}^p,\mathcal{S}^\infty$ and so on. For each $Z\in BMO$, we set
 \[
  \mathscr{Exp}(Z\cdot B)_0^t=\exp\left(\int^t_0 Z_s dB_s-\frac{1}{2}\int^t_0|Z_s|^2ds\right),
 \]
which is a martingale  by
 \cite{K}. Thus it follows from Girsanov's theorem that  $(B_t-\int_{0}^tZ_sds)_{0\leq t\leq T}$
is a Brownian motion under the equivalent probability measure $\mathscr{Exp} (Z\cdot B)_{0}^T d\mathbf{P}$.

\section{Lipschitz case}
In this section, we study the solvability of the  mean reflected BSDE \eqref{my1} with Lipschitz generator and $p$-integrable terminal condition.

\begin{definition}
By a deterministic solution to \eqref{my1}, we mean a  triple of progressively measurable processes $(Y, Z, K)\in\mathcal{S}^p\times\mathcal{H}^{p,d}\times\mathcal{A}$ such that \eqref{my1} holds  for some $p>1$.
\end{definition}

 \noindent In what follows,  we make use of the following  conditions on the terminal condition $\xi$, the generator $f$ and the running loss function $\ell$.

\begin{description}
\item[(H1)]  There exists $p>1$ such that $\xi \in \mathcal{L}^p$ with $\mathbf{E}[\ell(T,\xi)]\geq 0$.
\item[(H2)] The process $(f(t,0,{\delta_0},0))$ belongs to $\mathcal{H}^{p,1}$ and  there exists a constant $\lambda>0$ such that for any $t\in[0,T]$, $y_1,y_2\in \mathbb{R}$, $v_1,v_2\in \mathcal{P}_1(\mathbb{R})$, and $z_1,z_2\in \mathbb{R}^{d},$
\begin{equation*}
|f(t,y_1,v_1,z_1)-f(t,y_2,v_2,z_2)|\leq \lambda \left( |y_1-y_2|+{W_1(v_1,v_2)}+|z_1-z_2| \right).
\end{equation*}
\item[(H3)]  There exists a constant $L>0$ such that,
\begin{enumerate}
	\item $(t,y)\rightarrow \ell(t,y)$ is continuous,
	\item $\forall t\in[0,T]$, $y\rightarrow\ell(t,y)$ is strictly increasing,
	\item  $\forall t\in[0,T]$, $\mathbf{E}[\ell(t,\infty)]>0$,
	\item $\forall t\in[0,T]$, $\forall y\in\mathbb{R} $, $|\ell(t,y)| \leq L(1+|y|)$.
\end{enumerate}
\item[(H4)]  There exist two constants $\kappa>1$ and $C>0$ such that for each $t\in[0,T]$ and $y_1,y_2\in\mathbb{R}$,
 \begin{equation*}
C|y_1-y_2|\leq |\ell(t,y_1)-\ell(t,y_2)|\leq \kappa C|y_1-y_2|.
\end{equation*}
\end{description}

\medskip

\noindent
In order to study mean reflected BSDEs, we introduce 
the following map $L_t:\mathcal{L}^1\rightarrow\mathbb{R}$ for each $t\in[0,T]$:
\begin{align}\label{my10241}
 L_t(\eta)= \inf\{ x\geq 0 : \mathbf{E}[\ell(t,x+\eta)] \geq 0 \}, \ \forall \eta\in\mathcal{L}^1.
\end{align}
When assumption (H3) is satisfied, the map  
$X\mapsto L_t(X)$ is well-defined, see  \cite{BH}. 
In particular, $L_t(0)$ is continuous in $t$. Moreover, if assumption (H4) is also fulfilled, then for each $t\in[0,T]$,
 \begin{equation}\label{my399}
|L_t(\eta^1)-L_t(\eta^2)|\leq \kappa\mathbf{E}[|\eta^1-\eta^2|], \ \forall  \eta^1,\eta^2 \in \mathcal{L}^1.
\end{equation}

\begin{remark}\label{my1041}{\upshape
  Remark that one can use the map $X\rightarrow L_t(X)$ to construct the term $K$ via a standard BSDE involving the term $Y$,  which is crucial for solving the mean reflected BSDEs, see Lemma \ref{my1046}. 
}
  \end{remark}

\noindent We are now ready to state the main result of this section.
\begin{theorem}\label{my916}
Assume that \emph{(H1)}-\emph{(H4)} are fulfilled.  
Then the quadratic {mean reflected} BSDE \eqref{my1}  admits a unique
deterministic solution $(Y,Z,K)\in \mathcal{S}^p\times \mathcal{H}^{p,d}\times\mathcal{A}$.
\end{theorem}

\begin{remark}{\upshape
Using a fixed-point method, Briand et al. \cite{BH} established the well-posedness of mean field BSDEs \eqref{my1} in the case that $p=2$. Note that the driver furthermore depends on the distribution of the first component $Y$ of the solution in our framework.
}
\end{remark}

\noindent
In order to prove Theorem \ref{my916}, we introduce a representation result for the solution to the problem \eqref{my1}, which plays a key role in establishing the existence and uniqueness result.

\begin{lemma}\label{my1046} Suppose Assumptions \emph{(H1)}-\emph{(H3)} hold.  Let $(Y,Z,K)\in \mathcal{S}^p\times\mathcal{H}^{p,d}\times\mathcal{A}$ be a deterministic solution to the BSDE with mean reflection \eqref{my1}. Then, for each $t\in[0,T]$\[
(Y_t,Z_t,K_t)=\bigg(y_t+\sup\limits_{t\leq s\leq T}L_s(y_s),z_t,\sup_{0\leq s\leq T} L_s(y_s)- \sup_{t\leq s\leq T} L_s(y_s)\bigg)\]
where $(y,z)\in\mathcal{S}^p\times\mathcal{H}^{p,d}$ is the solution to the following BSDE with the driver $f^Y(s,z)=f(s,Y_s,\mathbf{P}_{Y_s},z)$ on the  time horizon $[0,T]$: 
\begin{align}\label{myq1042}
y_t=\xi+\int^{T}_t f^Y(s,z_s)ds-\int^{T}_t z_s dB_s.
\end{align}
\end{lemma}
\begin{proof}
It follows from \cite[Theorem 4.2]{BD} that the BSDE \eqref{myq1042}
admits a unique solution $(y,z)\in\mathcal{S}^p\times\mathcal{H}^{p,d}$. Since $K$ is a deterministic process, $(Y_{\cdot}-(K_T-K_{\cdot}),Z_{\cdot})$ is again a $\mathcal{S}^p\times\mathcal{H}^{p,d}$-solution to the BSDE \eqref{myq1042}, which  implies that
\[
(Y_t,Z_t)=(y_t+K_T-K_t,z_t),\ \forall t\in[0,T].
\]
On the other hand, $(Y,Z,K)\in \mathcal{S}^p\times\mathcal{H}^{p,d}\times\mathcal{A}$ can also be regarded as a deterministic solution to the following mean reflected BSDE with fixed generator $f(\cdot,Y_{\cdot},\mathbf{P}_{Y_{\cdot}},Z_{\cdot})$
\begin{align*}
\begin{cases}
&\widetilde{Y}_t=\xi+\int_t^T f(s,Y_s,\mathbf{P}_{Y_s}, Z_s)ds-\int_t^T \widetilde{Z}_sdB_s+\widetilde{K}_T-\widetilde{K}_t, \quad 0\le t\le T,\\
& \mathbf{E}[\ell(t,\widetilde{Y}_t)]\geq 0,\quad \forall t\in [0,T] \,\,\mbox{ and }\,\, \int_0^T\mathbf{E}[\ell(t,\widetilde{Y}_t)]d\widetilde{K}_t = 0.
\end{cases}
\end{align*}
Note that $y_t=\mathbf{E}_t\big[\xi+\int^{T}_t f(s,Y_s,\mathbf{P}_{Y_s}, Z_s)ds\big]$ for any $t\in[0,T]$.
Thus, recalling \cite[Proposition 7]{BH},  we have  $
	{K}_t= \sup_{0\leq s\leq T} L_s(y_s)- \sup_{t\leq s\leq T} L_s(y_s)
$ for each $t\in[0,T]$. This concludes the proof.
\end{proof}

\medskip
\noindent
Next, we  use a linearization technique and a fixed-point argument to get existence and uniqueness of the solution. For this purpose, we need to introduce the following solution map $\Gamma$ defined for $U\in \mathcal{S}^p$ by $\Gamma(U)=Y$ where $Y$ is the first component of the solution $(Y,Z,K)$ to the following mean reflected BSDE with driver $f^U$:
\begin{align}\label{myq527}
\begin{cases}
&Y_t=\xi+\int_t^T f^U(s,Z_s)ds-\int_t^T Z_sdB_s+K_T-K_t, \quad 0\le t\le T,\\
 & \mathbf{E}[\ell(t,{Y}_t)]\geq 0,\quad \forall t\in [0,T] \,\,\mbox{ and }\,\, \int_0^T\mathbf{E}[\ell(t,{Y}_t)]d{K}_t = 0.
\end{cases}
\end{align}

\begin{lemma}\label{myq1792}
	Assume that \emph{(H1)}--\emph{(H3)} are satisfied and  $U\in\mathcal{S}^{p}$. Then, the mean reflected BSDE
\eqref{myq527} admits a unique solution $(Y,Z,K)\in \mathcal{S}^{p}\times \mathcal{H}^{p,d} \times\mathcal{A}$.
\end{lemma}
\begin{proof} The uniqueness is immediate from Lemma \ref{my1046}.
It follows from  \cite[Theorem 4.2]{BD} that BSDE \eqref{myq1042}
with the driver $f(\cdot,{U_{\cdot}},\mathbf{P}_{U_{\cdot}},z)$ 
has a unique solution $(y,z)\in\mathcal{S}^p\times\mathcal{H}^{p,d}$. Then in view of \cite[Proposition 7]{BH}, we obtain that the following mean reflected BSDE  with fixed generator $f(\cdot,{U_{\cdot}},\mathbf{P}_{U_{\cdot}},z_{\cdot})$
\begin{align}\label{my1049}
\begin{cases}
&\widetilde{Y}_t=\xi+\int_t^T f(s,U_s,\mathbf{P}_{U_s}, z_s)ds-\int_t^T \widetilde{Z}_sdB_s+\widetilde{K}_T-\widetilde{K}_t, \quad 0\le t\le T,\\
& \mathbf{E}[\ell(t,\widetilde{Y}_t)]\geq 0,\quad \forall t\in [0,T] \,\,\mbox{ and }\,\, \int_0^T\mathbf{E}[\ell(t,\widetilde{Y}_t)]d\widetilde{K}_t = 0
\end{cases}
\end{align} has a unique solution $(Y,Z,K)\in\mathcal{S}^p\times\mathcal{H}^{p,d}\times\mathcal{A}.$ 
In light of the representation result (Lemma \ref{my1046}), we get $z\equiv Z$ and thus $(Y,Z,K)\in\mathcal{S}^p\times\mathcal{H}^{p,d}\times\mathcal{A}$ is the  solution to the mean reflected BSDE \eqref{myq527}.
\end{proof}

\medskip
\noindent
Let us now prove that $\Gamma$ defines a contraction map on a small time interval $[T-h, T]$, in which $h$ is to be determined later.
Note that in the spirit of Lemma \ref{myq1792}, we have  $\Gamma\left(\mathcal{S}^p_{[T-h, T]} \right)\subset \mathcal{S}^p_{[T-h, T]}$ for any $h\in (0,T]$.
\begin{lemma}\label{ref.3.5}
Assume that \emph{(H1)}--\emph{(H4)} hold.  Then there exists a constant $\delta>0$ depending only on $p,\lambda$ and $\kappa$  such that for any $h\in(0,\delta]$, the {mean reflected} BSDE \eqref{my1} admits a unique solution $(Y,Z,K)~\in~ \mathcal{S}^{p}_{[T-h,T]}\times \mathcal{H}^{p,d}_{[T-h,T]}\times\mathcal{A}^p_{[T-h,T]}$ on the time interval $[T-h,T]$.
\end{lemma}
\begin{proof} 
The proof is divided into three steps.
\medskip

\noindent {\bf Step 1 (A priori estimate).}
The main idea is similar to  \cite[Lemma 2.8]{HMW} and we give the sketch of the proof for readers' convenience. Let $U^i\in \mathcal{S}^p_{[T-h,T]}$, for $i=1,2$.
It follows from Lemma \ref{my1046} that
\begin{align}\label{my1001}
\Gamma(U^i)_t:=y^i_t+\sup\limits_{t\leq s\leq T}L_s(y^i_s), \ \ \forall t\in[T-h,T],
\end{align}
where $y^{i}$ is the solution to the  BSDE \eqref{myq1042} with driver $f^{U^i}$ and terminal condition $\xi$.
For each $t\in[0,T]$, we denote \[
\beta_t=\frac{f^{U^1}(t,z^{1}_t)-f^{U^1}(t,z^{2}_t)}{|z^{1}_t-z^{2}_t|^2}(z^{1}_t-z^{2}_t)\mathbf{1}_{\{|z^{1}_t-z^{2}_t|\neq 0\}}.
\]
Then, the pair of processes $(y^{1}-y^{2},z^{1}-z^{2})$ solves the following BSDE:
{ \begin{align*}
y^{1}_t-y^{2}_t = \int^T_t\left(\beta_s(z^{1}_s-z^{2}_s)^{\top}+f^{U^1}(s,z^{2}_s)-f^{U^2}(s,z^{2}_s)\right)ds-\int^T_t(z^{1}_s-z^{2}_s)dB_s.
\end{align*}
}

\noindent Since $\widetilde {B}_t:=B_t-\int_{0}^t\beta^{\top}_sds$ defines a Brownian motion under the equivalent probability measure $\widetilde{\mathbf{P}}$ given by
$d\widetilde{\mathbf{P}}: = \mathscr{Exp} (\beta\cdot B)_{0}^Td\mathbf{P}$, it follows that for every $t\in[0,T]$,
\begin{align*}
y^{1}_t-y^{2}_t=\mathbf{E}_t\bigg[\mathscr{Exp} (\beta\cdot B)_{t}^T\bigg(\int^{T}_t \left(f^{U^1}(s,z^{2}_s)-f^{U^2}(s,z^{2}_s)\right)ds\bigg)\bigg].
\end{align*}
Applying  H\"{o}lder's inequality yields, for any $\mu\in (1,p)$ and any $t\in[T-h,T]$,
\begin{align}\label{myq210}
&|y^{1}_t-y^{2}_t|\leq \exp\bigg(\frac{\lambda^2h}{2(\mu-1)}\bigg)\lambda h \mathbf{E}_t\bigg[\bigg(\sup\limits_{s\in[T-h,T]}|U_s^1-U^2_s|+\sup\limits_{s\in[T-h,T]}\mathbf{E}\left[|U_s^1-U^2_s|\right]\bigg)^\mu\bigg]^{\frac{1}{\mu}}.
\end{align}

\noindent {\bf Step 2 (The contraction).}
Recalling \eqref{my1001} and \eqref{my399}, we have
\begin{align*}\begin{split}
\sup\limits_{s\in[T-h,T]}|\Gamma(U^1)_s-\Gamma(U^2)_s|^p\leq 2^{p-1}\left(\sup\limits_{s\in[T-h,T]}|y^1_s-y^2_s|^p+\kappa^p\sup\limits_{s\in[T-h,T]}\mathbf{E}\left[|y^1_s-y^2_s|^p\right]\right).
\end{split}
\end{align*}
Recalling \eqref{myq210} and
applying {Doob's maximal inequality}, we derive
\begin{align*}
&\mathbf{E}\bigg[\sup\limits_{t\in[T-h,T]}|\Gamma(U^1)_t-\Gamma(U^2)_t|^p\bigg]\leq 2^{p-1}(1+\kappa^p)\lambda^p h^p
\exp\bigg(\frac{p\lambda^2h}{2(\mu-1)}\bigg)\\
&\hspace*{2cm}\times \bigg(\frac{p}{p-\mu}\bigg)^{\frac{p}{\mu}}\mathbf{E}\bigg[\bigg(\sup\limits_{s\in[T-h,T]}|U_s^1-U^2_s|+\sup\limits_{s\in[T-h,T]}\mathbf{E}\left[|U_s^1-U^2_s|\right]\bigg)^p\bigg].
\end{align*}

\noindent Consequently, for any $\mu\in(1,p)$ and $h\in(0,\mu-1]$, we have
 \begin{align*}
\mathbf{E}\bigg[\sup\limits_{t\in[T-h,T]}|\Gamma(U^1)_t-\Gamma(U^2)_t|^p\bigg]^{\frac{1}{p}}\leq \Lambda(\mu)
\mathbf{E}\bigg[\sup\limits_{s\in[T-h,T]}|U^1_s-U^2_s|^p\bigg]^{\frac{1}{p}}
\end{align*}
with
\[
\Lambda(\mu)=4(1+\kappa)\lambda \exp\left(\frac{\lambda^2}{2}\right) \bigg(\frac{p}{p-\mu}\bigg)^{\frac{1}{\mu}} (\mu-1).
\]
 Then we choose  a small enough constant $\mu^*\in(1,p)$ depending only on $p,\lambda$ and $\kappa$ such that $\Lambda (\mu^*) < 1$
 and set ${\delta}:=\mu^*-1.$
It follows that  $\Gamma$ is a contraction map on the time interval $[T-h,T]$ for any $h\in(0,\delta]$. 

\medskip
\noindent {\bf Step 3 (Uniqueness and existence).} 
The uniqueness is immediate from the fact that any solution  to the {mean reflected} BSDE \eqref{my1}  is a fixed point of the map $\Gamma$. For any $h\in(0,\delta]$, the function $\Gamma$ has a unique fixed point $Y\in\mathcal{S}^p_{[T-h,T]}$. Then the mean reflected BSDE \eqref{myq527} with driver $f^Y$ admits a unique solution
$(\widetilde{Y},{Z},{K})\in \mathcal{S}^{p}_{[T-h,T]}\times \mathcal{H}^{p,d}_{[T-h,T]}\times\mathcal{A}_{[T-h,T]}$. It immediately follows that $\widetilde{Y}=\Gamma(Y)=Y$, so $(Y,Z,K)$ is the desired solution to the {mean reflected} BSDE \eqref{my1}  on the time interval $[T-h,T].$  This completes the proof.
\end{proof}

\bigskip\noindent We now prove the main result with the help of the intermediate lemmas above.\medskip

\begin{proof}[Proof of Theorem \ref{my916}] 
Note that the length of the time interval on which the map $\Gamma$ is contractive depends only on $p,\kappa$ and $\lambda$. By a standard BSDE approach, we split the arbitrary time interval $[0,T]$ into a finite number of small time intervals. On each small time interval, we can then apply Lemma \ref{ref.3.5} to get a local solution. A global solution on the whole
time interval is obtained by stitching the local ones.
The global uniqueness on $[0, T]$ follows from the local uniqueness on each small time interval.  The proof is  complete.
\end{proof}

\begin{remark}
{\upshape 
By more involved and delicate estimates, our method can still be applied to study Lipschitz mean reflected BSDEs when the driver depends on the distribution of $Z$ as well. However, it gets much more complicated for the quadratic case, which needs further study (even for quadratic mean-field BSDEs, see, e.g. \cite{HH2}).
}
\end{remark}

\section{Bounded terminal condition}
In this section, we combine BMO martingale theory and a fixed-point argument in order to analyze the quadratic mean reflected BSDE \eqref{my1} with bounded terminal condition. 

\medskip

\noindent In what follows,  we make use of the following  conditions on the terminal condition $\xi$ and the driver $f$.
\begin{description}
\item[(H1')]  The terminal condition $\xi \in \mathcal{L}^{\infty}$ with $\mathbf{E}[\ell(T,\xi)]\geq 0$.
\item[(H2')]  The process $(f(t,0,\delta_0,0))$ is uniformly bounded and there exist two positive  constants $\beta$ and $\gamma$ such that  for any $t\in[0,T]$, $y_1,y_2\in \mathbb{R}$, $v_1,v_2\in \mathcal{P}_1(\mathbb{R})$, and $z_1,z_2\in \mathbb{R}^{d},$
\[
|f(t,y_1,v_1,z_1)-f(t,y_2,v_2,z_2)|\leq \beta \left( |y_1-y_2|+{W_1(v_1,v_2)} \right)+\gamma(1+|z_1|+|z_2))|z_1-z_2|.\]
\end{description}

\noindent We are now ready to state the main result of this section.
\begin{theorem}\label{my1162}
Assume that \emph{(H1')}, \emph{(H2')}, \emph{(H3)} and \emph{(H4)} are satisfied. Then the quadratic  BSDE~\eqref{my1} with mean reflection admits a unique solution $(Y,Z,K)\in \mathcal{S}^{\infty}\times BMO\times\mathcal{A}$.
\end{theorem}

\begin{remark}\label{my1996}{\upshape In view of \cite[Theorem 7.3.3]{zhang2017} and \cite[Theorem 3.1]{HH1},  Lemma \ref{my1046} still holds under conditions {(H1')}, {(H2')} and 
{(H3)}. 
When the driver does not depend on the
 distribution of $Y$, the authors of \cite{HH1} proved that quadratic mean reflected BSDEs  admits a unique solution. Compared with that of \cite{HH1}, we apply Lemma \ref{my1046} to remove the following additional assumption:
\[
\text{ $(t,y)\mapsto f(t,y,0)$ is uniformly  bounded.}
\]
}
\end{remark}

\medskip  
\noindent 
As in the Lipschitz case, we will prove that the solution map $\Gamma$ defines a contraction map.

\begin{lemma}\label{myq792}
	Assume that \emph{(H1')}, \emph{(H2')} and \emph{(H3)}  are satisfied and  $U\in\mathcal{S}^{\infty}$. Then, the quadratic  BSDE
\eqref{myq527} with mean reflection, with driver $f^U$, admits a unique solution $(Y,Z,K)\in \mathcal{S}^{\infty}\times BMO \times\mathcal{A}$.
\end{lemma}
\begin{proof}  
It follows from  \cite[Theorem 7.3.3]{zhang2017} that the quadratic BSDE \eqref{myq1042}
with the driver $f(\cdot,{U_{\cdot}},\mathbf{P}_{U_{\cdot}},z)$ 
has a unique solution $(y,z)\in\mathcal{S}^{\infty}\times BMO$. Then with the help of \cite[Theorem 3.1]{HH1}, we have that  the quadratic mean reflected BSDE \eqref{my1049} with the fixed generator $f(\cdot,{U_{\cdot}},\mathbf{P}_{U_{\cdot}},z_{\cdot})$ has a unique solution $(Y,Z,K)\in\mathcal{S}^{\infty}\times BMO\times\mathcal{A}.$  Recalling Remark \ref{my1996} and Lemma \ref{my1046}, we derive that  $z\equiv Z$ and  $(Y,Z,K)\in\mathcal{S}^{\infty}\times BMO\times\mathcal{A}$ is the  solution to the mean reflected BSDE \eqref{myq527}. The uniqueness eventually follows from  Lemma \ref{my1046}, which ends the proof.
\end{proof}

\medskip
\noindent We are now ready to state the proof of the main result of this section.\medskip

\begin{proof}[Proof of Theorem \ref{my1162}]
Let $U^i\in \mathcal{S}^{\infty}$, $i=1,2$.
It follows from Lemma \ref{my1046} that
\begin{align}\label{my1002}
\Gamma(U^i)_t:=y^i_t+\sup\limits_{t\leq s\leq T}L_s(y^i_s), \ \ \forall t\in[T-h,T],
\end{align}
where $y^{i}$ is the solution to the quadratic BSDE \eqref{myq1042} with driver $f^{U^i}$ and the terminal condition $\xi$. Following the proof of Lemma \ref{ref.3.5} step by step (noting that $(\beta_t) \in BMO$ in this case), we conclude that for any $t\in[0,T]$,
{ \begin{align*}
y^{1}_t-y^{2}_t=\mathbf{E}^{\widetilde{\mathbf{P}}}_t\bigg[\int^{T}_t \left(f^{U^1}(s,z^{2}_s)-f^{U^2}(s,z^{2}_s)\right)ds\bigg],
\end{align*}}
which together with Assumption (H2')  implies that for any $t\in[T-h,T]$,
 \begin{align*}
|y^{1}_t-y^{2}_t|\leq \beta h\|U^1-U^2\|_{\mathcal{S}^{\infty}_{[T-h,T]}}+\beta h\sup\limits_{s\in[T-h,T]}\mathbf{E}[|U_s^1-U^2_s|].
\end{align*}
In view of \eqref{my1002} and \eqref{my399}, we again derive that
 \begin{align*}\begin{split}
\|\Gamma(U^1)-\Gamma(U^2)\|_{\mathcal{S}^{\infty}_{[T-h,T]}}
\leq 2(1+\kappa)\beta h \|U^1-U^2\|_{\mathcal{S}^{\infty}_{[T-h,T]}}.
\end{split}
\end{align*}
Then we can find a small enough constant $h$ depending only on $\beta$ and $\kappa$ such that $2(1+\kappa)\beta h< 1.$ Therefore, $\Gamma$ defines a contraction map on the time interval $[T-h,T]$.
Proceeding exactly as in Theorem \ref{my916} and Lemma \ref{ref.3.5}, we complete the proof.
\end{proof}

\medskip
\noindent
Note that the process $(\beta_t)$ may be unbounded in the BMO space  and then  the fixed-point argument fails to work when the terminal condition is  unbounded.  In the next section, we make use of a $\theta$-method to overcome this difficulty under the further assumption of either convexity or concavity on the generator.

\section{Unbounded terminal condition}

In this section, we investigate the solvability of the  mean reflected BSDE \eqref{my1} with quadratic generator $f$ and unbounded terminal value $\xi$. In what follows,  we make use of the following conditions on the parameters $\xi$ and $f$.

\begin{description}
\item[(H1'')]  The terminal condition $\xi \in \mathbb{L}$ with $\mathbf{E}[\ell(T,\xi)]\geq 0$.
\item[(H2'')]  There exist a positive progressively measurable process $(\alpha_{t})_{0\leq t\leq T}$ with   $\int^T_0\alpha_tdt\in\mathbb{L}$  and two positive  constants $\beta$ and $\gamma$ such that 
\begin{enumerate}
	\item $\forall (t,y,v,z)\in[0,T]\times \mathbb{R}\times \mathcal{P}_1(\mathbb{R}) \times\mathbb{R}^{d}$,
$|f(t,y,v,z)|\leq \alpha_t+\beta (|y|+W_1(v,\delta_0))+\frac{\gamma}{2}|z|^2,
$
\item  $\forall t\in[0,T]$, $y_1,y_2\in \mathbb{R}$, $v_1,v_2\in \mathcal{P}_1(\mathbb{R})$, $ z\in \mathbb{R}^{d}$,
$$
|f(t,y_1,v_1,z)-f(t,y_2,v_2,z)|\leq \beta (|y_1-y_2|+W_1(v_1,v_2)),$$
\item  $\forall (t,y,v)\in[0,T]\times \mathbb{R}\times \mathcal{P}_1(\mathbb{R})$, $z\rightarrow f(t,y,v,z)$ is convex or concave.
\end{enumerate}
\end{description}

\noindent Let us now state the main result of this section.
\begin{theorem}\label{my116}
Assume that \emph{(H1'')}, \emph{(H2'')}, \emph{(H3)} and \emph{(H4)} are fulfilled.  
Then the quadratic {mean reflected} BSDE \eqref{my1}  admits a unique
deterministic solution $(Y,Z,K)\in \mathbb{S}\times \mathcal{H}^{d}\times\mathcal{A}$.
\end{theorem}

\begin{remark}\label{my3996}{\upshape Note that  in view of \cite[Corollary 6]{BH2008} and \cite[Proposition 7]{BH}, the representation result in Lemma \ref{my1046} still holds under {(H1'')}, {(H2'')} and 
{(H3)}. 
}
\end{remark}

\noindent In order to prove Theorem \ref{my116}, we need to recall some technical results on quadratic BSDEs. Consider the following  standard BSDE on the  time horizon $[0,T]$
\begin{align}\label{myq2}
y_t=\eta+\int^{T}_t g(s,z_s)ds-\int^{T}_t z_s dB_s.
\end{align}

\noindent The following result is important for our subsequent computations, and can be found in \cite[Lemmas A3 and A4]{FHT2}.
\begin{lemma}\label{my7}
Assume that $(y,z)\in \mathcal{S}^2\times\mathcal{H}^{2,d}$  is a solution to \eqref{myq2}. Suppose that there is a  constant $p\geq 1$ such that
\begin{align*}
\mathbf{E}\bigg[\exp\bigg\{2p\gamma \sup\limits_{t\in[0,T]}|y_t|+2p\gamma\int^T_0\alpha_tdt\bigg\}\bigg]<\infty.
\end{align*}
Then, we have
	\begin{description}
		\item[(i)] If $|g(t,z)|\leq \alpha_t+\frac{\gamma}{2}|z|^2$, then  for each $t\in[0,T]$,
\begin{align*}\exp\left\{p\gamma|y_t|\right\}\leq   \mathbf{E}_t\bigg[\exp\bigg\{p\gamma|\eta|+p\gamma\int^T_t\alpha_s ds\bigg\}	 \bigg].
		\end{align*}
				\item[(ii)] If $g(t,z)\leq \alpha_t+\frac{\gamma}{2}|z|^2$, then for each $t\in[0,T]$,
 \begin{align*}
		\exp\left\{{p\gamma}(y_t)^+\right\}\leq   \mathbf{E}_t\bigg[\exp\bigg\{p\gamma \eta^++p\gamma\int^T_t\alpha_s ds\bigg\}
		\bigg].
		\end{align*}	
	\end{description}
\end{lemma}

\medskip
\noindent We are now ready to combine the $\theta$-method and the representation result to prove Theorem \ref{my116}. In order to illustrate the main idea, we first deal with the uniqueness. 
\medskip

\begin{lemma}\label{myq7900}
	Assume that all the conditions of Theorem \ref{my116} are satisfied. Then, the  quadratic mean reflected BSDE \eqref{my1} has at most one deterministic solution $(Y,Z,K)\in\mathbb{S}\times \mathcal{H}^{d}\times\mathcal{A}$.
\end{lemma}
\begin{proof}
\noindent For $i=1,2$,
let $({Y}^i,{Z}^i,{K}^i)$ be a  deterministic $ \mathbb{S}\times\mathcal{H}^d\times\mathcal{A}$-solution to the quadratic mean reflected BSDE~\eqref{my1}. From  Lemma \ref{my1046} and Remark \ref{my3996}, we have
\begin{align}\label{myq91}
Y^{i}_t:=y^{i}_t+\sup\limits_{t\leq s\leq T}L_s(y^i_s), \ \ \forall t\in[0,T],
\end{align}
where $(y^{i},z^i)\in\mathbb{S}\times\mathcal{H}^d$ is the solution to the following quadratic BSDE:
\begin{align*}
y_t^i=\xi+\int^{T}_t f\left(s,Y^i_s,\mathbf{P}_{Y^i_s},z^i_s\right)ds-\int^{T}_t z^i_s dB_s.
\end{align*}

\noindent Assume without loss of generality that $f(t,y,v,\cdot)$ is concave (see Remark \ref{myrk11}). For each $\theta\in (0,1)$, we denote 
\[ \delta_{\theta}\ell=\frac{\theta \ell^1-\ell^2}{1-\theta}, \ \delta_{\theta}\widetilde{\ell}=\frac{\theta \ell^2-\ell^1}{1-\theta} ~~\text{and} ~~ \delta_{\theta}\overline{\ell}:= |\delta_{\theta}\ell|+|\delta_{\theta}\widetilde{\ell}|\]
for $\ell=Y,y$ and $z$. Then, the pair of processes $(\delta_{\theta}y,\delta_{\theta}z)$
satisfies the following BSDE:
 \begin{align}\label{myq123}
\begin{split}
\delta_{\theta}y_t=&-\xi+\int^T_t\left(\delta_{\theta}f(s,\delta_{\theta}z_s)+\delta_{\theta}f_0(s)\right)ds-\int^T_t\delta_{\theta}z_sdB_s,
\end{split}
\end{align}
where the generator is given by
{\small \begin{align*}
 & \delta_{\theta}f_0(t)=
\frac{1}{1-\theta}\left(f(t,Y^1_{t}, \mathbf{P}_{Y^1_{t}}, z^{2}_t)-f(t,Y^2_{t},\mathbf{P}_{Y^2_{t}}, z^{2}_t)\right)\\
&
\delta_{\theta}f(t,z)=\frac{1}{1-\theta}\left(
\theta f(t,Y^1_t, \mathbf{P}_{Y^1_{t}}, z^{1}_t)- f(t,Y^1_t,\mathbf{P}_{Y^1_{t}}, -(1-\theta)z+\theta z^{1}_t)\right).
\end{align*}}

\noindent Recalling assumptions (H2''), we have that
\begin{align*}
 &\delta_{\theta}f_0(t)\leq \beta\big( |Y^1_t|+|\delta_{\theta}Y_t|+\mathbf{E}[|Y^1_t|+|\delta_{\theta}Y_t|]\big), \\  &\delta_{\theta}f(t,z)
 \leq -f(t,Y^1_t, \mathbf{P}_{Y^1_{t}}, -z)
 \leq \alpha_t+\beta\big(|Y^1_t|+\mathbf{E}[|Y^1_t|]\big)+\frac{\gamma}{2}|z|^2.
\end{align*}
Set $C_1:=\sup\limits_{s\in[0,T]}\mathbf{E}[|Y^{1}_{s}|+|Y^{2}_{s}|]$ and
{ \begin{align*}
&\chi= \int^T_0\alpha_sds+2\beta C_1 T+2\beta T\bigg(\sup\limits_{s\in[0,T]}|Y^{1}_{s}|+\sup\limits_{s\in[0,T]}|Y^{2}_{s}|\bigg),\\
&\widetilde{\chi}=\int^T_0\alpha_sds+2\beta C_1 T+2\beta T\bigg(\sup\limits_{s\in[0,T]}|Y^{1}_{s}|+\sup\limits_{s\in[0,T]}|Y^{2}_{s}|\bigg)+\sup\limits_{s\in[0,T]}|y^{1}_{s}|+\sup\limits_{s\in[0,T]}|y^{2}_{s}|.\end{align*}
}

\noindent Using assertion (ii) of Lemma \ref{my7} to  \eqref{myq123}, we derive that  for any $p\geq 1$,
{ \begin{align*}	
 \begin{split}
 \exp\left\{{p\gamma}\big(\delta_{\theta}y_t\big)^+\right\}\leq   \mathbf{E}_t\left[\exp\bigg\{p\gamma \bigg(|\xi|+\chi+\beta (T-t)\bigg(\sup\limits_{s\in[t,T]}|\delta_{\theta}Y_{s}|+\sup\limits_{s\in[t,T]}\mathbf{E}[|\delta_{\theta}Y_{s}|] \bigg)\bigg)\bigg\} \right].
 \end{split}
	\end{align*}
}
\noindent Similarly, we have
{ \begin{align*}	
 \begin{split}
 \exp\left\{{p\gamma}\left(\delta_{\theta}\widetilde{y}_t\right)^+\right\}\leq   \mathbf{E}_t\left[\exp\bigg\{p\gamma \bigg(|\xi|+\chi+\beta (T-t)\bigg(\sup\limits_{s\in[t,T]}|\delta_{\theta}\widetilde{Y}_{s}|+\sup\limits_{s\in[t,T]}\mathbf{E}[|\delta_{\theta}\widetilde{Y}_{s}|]\bigg)\bigg)\bigg\}	\right].
 \end{split}
	\end{align*}
}

\noindent In view of the fact that
\begin{align*}
\left(\delta_{\theta}{y}\right)^-
\leq \left(\delta_{\theta}\widetilde{y}\right)^++2|y^{2}| \ \text{and}\ \left(\delta_{\theta}\widetilde{y}\right)^-
\leq \left(\delta_{\theta}{y}\right)^++2|y^{1}|,
\end{align*}
we have
{  \begin{align*}
	\begin{split}
	\exp\left\{p\gamma \left|\delta_{\theta}{y}_t\right|\right\}\vee \exp\left\{p\gamma \left|\delta_{\theta}\widetilde{y}_t\right|\right\} &\leq
	\exp\left\{{p\gamma}\left(\left( \delta_{\theta}{y}_t\right)^++\left(\delta_{\theta}\widetilde{y}_t\right)^++2|y^{1}_t|+2|y^{2}_t|\right)\right\}\\
	& \hspace*{-1cm}\leq \mathbf{E}_t\left[\exp\bigg\{p\gamma \bigg(|\xi|+\widetilde{\chi}+\beta (T-t)\bigg(\sup\limits_{s\in[t,T]}\delta_{\theta}\overline{Y}_{s}+\sup\limits_{s\in[t,T]}\mathbf{E}[\delta_{\theta}\overline{Y}_{s}]\bigg)\bigg)\bigg\}\right]^2.
	\end{split}
	\end{align*}
}
Applying {Doob's maximal inequality} and
H\"{o}lder's inequality, we get  that for each  $p\geq 1$ and $t\in[0,T]$,
{\small  \begin{align}\label{myq2231}
	\begin{split}
	\mathbf{{E}}\left[\exp\bigg\{p\gamma \sup\limits_{s\in[t,T]}\delta_{\theta}\overline{y}_s\bigg\}\right]&\leq \mathbf{{E}}\left[\exp\bigg\{p\gamma \sup\limits_{s\in[t,T]}|\delta_{\theta}{y}_s|\bigg\}\exp\bigg\{p\gamma \sup\limits_{s\in[t,T]}|\delta_{\theta}\widetilde{y}_s|\bigg\}\right]
	\\ &\leq  4  \mathbf{E}\left[\exp\bigg\{4p\gamma \bigg(|\xi|+\widetilde{\chi} +\beta (T-t)\bigg(\sup\limits_{s\in[t,T]}\delta_{\theta}\overline{Y}_{s}+\sup\limits_{s\in[t,T]}\mathbf{E}[\delta_{\theta}\overline{Y}_{s}]\bigg) \bigg)\bigg\}	\right].
	\end{split}
	\end{align}
}
Set $ C_2:=\sup\limits_{0\leq s\leq T}|L_s(0)|
+2\kappa \sup\limits_{s\in[0,T]}\mathbf{E}\big[|y^{1}_{s}|+|y^{2}_{s}|\big]$. Recalling \eqref{myq91} and assumption (H4), we derive that
\begin{align*}
|\delta_{\theta} Y_t|\leq C_2+\left|\delta_{\theta}{y}_t\right|+\kappa\sup\limits_{t\leq s\leq T}\mathbf{E}\left[|\delta_{\theta}y_s|\right]\  \text{and} \ |\delta_{\theta} \widetilde{Y}_t|\leq C_2+\left|\delta_{\theta}\widetilde{y}_t\right|+\kappa\sup\limits_{t\leq s\leq T}\mathbf{E}\left[|\delta_{\theta}\widetilde{y}_s|\right], \ \forall t\in[0,T],
\end{align*}
which together with  Jensen's inequality implies that for each  $p\geq 1$ and $t\in[0,T]$,
{\small  \begin{align}\label{myq22313}
	\begin{split}
&\mathbf{{E}}\left[\exp\big\{p\gamma \sup\limits_{s\in[t,T]}\delta_{\theta}\overline{Y}_s\big\}\right]\leq e^{2p\gamma C_2}\mathbf{{E}}\left[\exp\big\{p\gamma \sup\limits_{s\in[t,T]}\delta_{\theta}\bar{y}_s\big\}\right]\mathbf{{E}}\left[\exp\bigg\{2\kappa p\gamma\sup\limits_{s\in[t,T]}\delta_\theta\bar{y}_s\bigg\}\right] 
 \\
&\hspace*{1cm}\leq  e^{2p\gamma C_2}\mathbf{{E}}\left[\exp\big\{(2+4\kappa)p\gamma \sup\limits_{s\in[t,T]}\delta_{\theta}\bar{y}_s\big\}\right]\\
&\hspace*{1cm}\leq  4 \mathbf{E}\left[\exp\bigg\{(8+16\kappa)p\gamma \bigg(|\xi|+\widetilde{\chi}+C_2 +\beta (T-t)\big(\sup\limits_{s\in[t,T]}\delta_{\theta}\overline{Y}_{s}+\sup\limits_{s\in[t,T]}\mathbf{E}[\delta_{\theta}\overline{Y}_{s}]\big) \bigg)\bigg\}	\right]
\\
&\hspace*{1cm}\leq  4 \mathbf{E}\left[\exp\bigg\{(8+16\kappa)p\gamma \bigg(|\xi|+\widetilde{\chi}+C_2 +\beta (T-t)\sup\limits_{s\in[t,T]}\delta_{\theta}\overline{Y}_{s} \bigg)\bigg\}	\right] \\
& \hspace*{3cm} \times \mathbf{E}\left[\exp\bigg\{(8+16\kappa)p\gamma\beta (T-t) \sup\limits_{s\in[t,T]}\delta_{\theta}\overline{Y}_{s}\bigg\}	\right],
	\end{split}
	\end{align}
}where we have used \eqref{myq2231} in the third inequality.

\medskip \noindent  Choose a  constant $h\in (0,T]$ depending only on $\beta$ and $\kappa$ such that $ (16+32\kappa)\beta h <1 $.
In the spirit of H\"{o}lder's inequality, we derive that for any  $p\geq 1$,
\begin{align*}	
	\begin{split}
	&\mathbf{{E}}\bigg[\exp\big\{p\gamma \sup\limits_{s\in[T-h,T]}\delta_{\theta}\overline{Y}_s\big\}\bigg]\\&\leq 4\mathbf{{E}}\bigg[\exp\bigg\{{(16+32\kappa) p\gamma}(|\xi|+\widetilde{\chi}+C_2)\bigg\}		 \bigg]^{\frac{1}{2}}  \mathbf{E}\bigg[\exp\bigg\{(16+32\kappa)\beta h p\gamma \sup\limits_{s\in[t,T]}\delta_{\theta}\overline{Y}_{s}\bigg\}	\bigg]\\
	&\leq 4\mathbf{{E}}\bigg[\exp\bigg\{{(16+32\kappa) p\gamma}(|\xi|+\widetilde{\chi}+C_2)\bigg\}		 \bigg]\mathbf{{E}}\bigg[\exp\bigg\{p\gamma\sup\limits_{s\in[T-h,T]}\delta_{\theta}\overline{Y}_{s}\bigg\}		 \bigg]^{(16+32\kappa) \beta h},
	\end{split}
	\end{align*}
which together with the fact that $(16+32\kappa)\beta h<1$ implies that for any $p\geq 1$ and $\theta\in(0,1)$
\[
\mathbf{{E}}\left[\exp\bigg\{p\gamma \sup\limits_{s\in[T-h,T]}\delta_{\theta}\overline{Y}_s\bigg\}\right]\leq \mathbf{{E}}\bigg[4\exp\bigg\{{(16+32\kappa) p\gamma}(|\xi|+\widetilde{\chi}+C_2)\bigg\}		 \bigg]^{\frac{1}{1-(16+32\kappa)\beta h}}<\infty.
\]
Note that
${Y}^{1}-{Y}^{2}= (1-\theta)(\delta_{\theta}{Y}+Y^{1}). $ It follows that
\begin{align*}
\mathbf{E}\bigg[\sup\limits_{t\in[T-h,T]}
\big|{Y}^{1}_t-{Y}^{2}_t\big|\bigg]\leq (1-\theta)\bigg(\frac{1}{\gamma}\sup\limits_{\theta\in(0,1)}\mathbf{{E}}\bigg[\exp\big\{\gamma \sup\limits_{s\in[T-h,T]}\delta_{\theta}\overline{Y}_s\big\}\bigg]+\mathbf{{E}}\bigg[\sup\limits_{t\in[0,T]}\big|{Y}^{1}_t\big|\bigg]\bigg).
\end{align*}
Letting $\theta\rightarrow 1$ yields $Y^1=Y^2$. Applying It\^o's formula to $\left|Y^1-Y^2\right|^2$ yields $(Z^1,K^1)=(Z^2,K^2)$  on $[T-h,T]$. The uniqueness of the  solution on the whole interval is inherited from the uniqueness on each small time interval. The proof is complete.
\end{proof}

\begin{remark}\label{myrk11}
    \noindent  {\upshape In the convex case, one should use $\ell^1-\theta \ell^{2}$ and   $\ell^{2}-\theta \ell^{1}$ instead of
 $\theta \ell^{1}- \ell^{2}$ and  $\theta \ell^{2}- \ell^{1}$  in the definition of $\delta_{\theta}\ell$ and $\delta_{\theta}\widetilde{\ell}$, respectively. Then the generator of  BSDE \eqref{myq123} satisfies
 \begin{align*}
 &\delta_{\theta}f_0(t)\leq \beta \big(|Y^2_t|+|\delta_{\theta}Y_t|+\mathbf{E}[(|Y^2_t|+|\delta_{\theta}Y_t|)]\big),\\
& \delta_{\theta}f(t,z)
 \leq f(t,Y^2_t,\mathbf{P}_{Y^2_t}, z)
 \leq \alpha_t+\beta \big(|Y^2_t|+\mathbf{E}[|Y^2_t|]\big)+\frac{\gamma}{2}|z|^2.
\end{align*}
One can check that \eqref{myq2231} and \eqref{myq22313} still hold in this context.
}
\end{remark}

\begin{remark}{\upshape Due to the presence of mean reflection,
one cannot directly apply the $\theta$-method to establish the desired estimates for  quadratic mean reflected BSDEs with unbounded terminal condition as in \cite{BH}. With the help of  Lemma \ref{my1046}, we could overcome this difficulty by analyzing a standard  quadratic BSDE. In a similar way,  \cite{HMW}  established the well-posedness of quadratic mean-field reflected BSDEs with unbounded terminal condition via nonlinear Snell envelope representation and quadratic BSDEs techniques.  }
\end{remark}

\noindent We now turn to the existence part of our result.
\begin{lemma}\label{myq79}
	Assume that all the conditions of Theorem \ref{my116} hold and $U\in\mathbb{S}$. Then, the quadratic mean reflected BSDE \eqref{myq527}, with driver $f^U$,
	admits a unique solution $(Y,Z,K)\in \mathbb{S}\times\mathcal{H}^d\times\mathcal{A}$.
\end{lemma}
\begin{proof} The uniqueness follows from  Lemma \ref{my1046} and Remark \ref{my3996}.
	In view of assumption (H2''), we have
	\begin{align}\label{myq82}
	|f(t,U_t,\mathbf{P}_{U_t},z)|\leq \alpha_t+\beta (|U_t|+\mathbf{E}\left[|U_t|\right])+\frac{\gamma}{2}|z|^2.
	\end{align}
\noindent It follows from \cite[Corollary 6]{BH2008} that the BSDE \eqref{myq527} admits a unique solution $(y,z)\in \mathbb{S}\times \mathcal{H}^d$. 
Then it follows from  \cite[Proposition 7]{BH} that the   mean reflected BSDE \eqref{my1049} with fixed driver $f(\cdot,{U_{\cdot}},\mathbf{P}_{U_{\cdot}},z_{\cdot})$
has a unique deterministic solution $({Y},{Z},{K})\in \mathbb{S}\times\mathcal{H}^{d}\times\mathcal{A}$. In the spirit of Lemma \ref{my1046} and Remark \ref{my3996}, we conclude that $z={Z}$, which implies that $(Y,Z,K)$ is the desired solution. This completes the proof.
\end{proof}
\medskip

\noindent
According to Lemma \ref{myq79}, we recursively define a sequence of stochastic processes $(Y^{(m)})_{m=1}^{\infty}$ through  the following quadratic  BSDE with mean reflection:
\begin{align}\label{myq78}
\begin{cases}
&Y_t^{(m)}=\xi+\int_t^T f(s,Y^{(m-1)}_s,\mathbf{P}_{Y^{(m-1)}_s},Z^{(m)}_s)ds-\int_t^T Z^{(m)}_sdB_s+K^{(m)}_T-K^{(m)}_t, \quad 0\le t\le T,\\
& \mathbf{E}[\ell(t,Y_t^{(m)})]\geq 0,\quad \forall t\in [0,T] \,\,\mbox{ and }\,\, \int_0^T \mathbf{E}[\ell(t,Y_t^{(m)})]dK^{(m)}_t = 0,
\end{cases}
\end{align}
where $Y^{(0)}\equiv 0$.
It is obvious that $(Y^{(m)},Z^{(m)},K^{(m)})\in \mathbb{S}\times\mathcal{H}^d\times\mathcal{A}$. 
\medskip

\noindent Next, we  apply  a $\theta$-method and BSDEs techniques to prove that $(Y^{(m)},Z^{(m)},K^{(m)})$ defines a Cauchy sequence: the corresponding limit is the  desired solution. We need the following technical results to complete the proof, whose proofs are postponed to the Appendix.

\begin{lemma}\label{myq7901}
	Assume that the conditions of Theorem \ref{my116} are fulfilled. Then, for any $p\geq 1$, we have
	\begin{align*}
	\begin{split}
\sup\limits_{m\geq 0}\mathbf{E}\left[\exp\left\{p\gamma\sup\limits_{s\in[0,T]}|Y^{(m)}_s|\right\}+\left(\int^T_0|Z^{(m)}_t|^2dt\right)^p+|K^{(m)}_T|\right]<\infty.
	\end{split}
	\end{align*}
\end{lemma}

\begin{lemma}\label{myq7902}
	Assume that all the conditions of Theorem \ref{my116} are satisfied. Then, for any $p\geq 1$, we have
	\begin{align*}
	\begin{split}
\Pi(p):=\sup\limits_{\theta\in(0,1)}\lim_{m\rightarrow \infty}\sup\limits_{q\geq 1}\mathbf{{E}}\left[\exp\left\{p\gamma \sup\limits_{s\in[0,T]}\delta_{\theta}\overline{Y}^{(m,q)}_s\right\}\right]<\infty,
	\end{split}
	\end{align*}
	where we use the following notations
	{ \[
\delta_{\theta}Y^{(m,q)}=\frac{\theta Y^{(m+q)}-Y^{m}}{1-\theta}, \ \delta_{\theta}\widetilde{Y}^{(m,q)}=\frac{\theta Y^{(m)}-Y^{(m+q)}}{1-\theta}~ \text{and} ~
\delta_{\theta}\overline{Y}:=|\delta_{\theta} Y^{(m,q)}|+|\delta_{\theta}\widetilde{Y}^{(m,q)}|.
\]
}
\end{lemma}
\medskip

 \noindent We are now in a position to complete the proof of the main result.\medskip

\begin{proof}[Proof of Theorem \ref{my116}] It suffices to prove the existence. Note that for any integer $p\geq 1$ and for any $\theta\in (0,1)$,
\begin{align*}
\limsup_{m\rightarrow \infty}\, \sup\limits_{q\geq 1}\mathbf{E}\left[\sup\limits_{t\in[0,T]}
\left|{Y}^{(m+q)}_t-{Y}^{(m)}_t\right|^p\right]\leq 2^{p-1}(1-\theta)^p\left(\frac{\Pi(1)p!}{\gamma^{p}}+\sup\limits_{m\geq 1}\mathbf{{E}}\bigg[\sup\limits_{t\in[0,T]}\big|{Y}^{(m)}_t\big|^p\bigg]\right),
\end{align*}
which together with Lemmas \ref{myq7901}, \ref{myq7902} and the arbitrariness of $\theta$ implies that
\begin{align*}
\limsup_{m\rightarrow \infty}\, \sup\limits_{q\geq 1}\mathbf{E}\left[\sup\limits_{t\in[0,T]}
\big|{Y}^{(m+q)}_t-{Y}^{(m)}_t\big|^p\right]=0, \ \forall p\geq 1.
\end{align*}
Applying It\^o's formula to $\big|{Y}^{(m+q)}_t-{Y}^{(m)}_t\big|^2$  yields
\begin{align*}
\begin{split}
&\mathbf{E}\left[\int^T_0\big|Z^{(m+q)}_t-Z^{(m)}_t\big|^2dt \right]\leq \mathbf{E}\left[\sup\limits_{t\in[0,T]}
\big|{Y}^{(m+q)}_t-{Y}^{(m)}_t\big|^2+
\sup\limits_{t\in[0,T]}
\big|{Y}^{(m+q)}_t-{Y}^{(m)}_t\big|\Delta^{(m,q)}\right]\\
&\leq \mathbf{E}\left[\sup\limits_{t\in[0,T]}
\big|{Y}^{(m+q)}_t-{Y}^{(m)}_t\big|^2\right]+
\mathbf{E}\left[|\Delta^{(m,q)}|^2\right]^{\frac{1}{2}}\mathbf{E}\left[\sup\limits_{t\in[0,T]}
\big|{Y}^{(m+q)}_t-{Y}^{(m)}_t\big|^2\right]^{\frac{1}{2}}
\end{split}
\end{align*}
with
\[
\Delta^{(m,q)}:=\int^T_0\left|f\left(t,Y^{(m+q-1)}_t,\mathbf{P}_{Y^{(m+q-1)}_t},Z^{(m+q)}_t\right)- f\left(t,Y^{(m-1)}_t,\mathbf{P}_{Y^{(m-1)}_t},Z^{(m)}_t\right)\right|dt+|K_T^{(m+q)}|+|K_T^{(m)}|.
\]
It follows from  Lemma \ref{myq7901} and dominated convergence theorem that  
\begin{align*}
\limsup_{m\rightarrow \infty}\, \sup\limits_{q\geq 1} \mathbf{E}\left[\left(\int^T_0\big|Z^{(m+q)}_t-Z^{(m)}_t\big|^2dt\right)^p\right]=0, \ \forall p\geq 1.
\end{align*}
Therefore, there exists a pair of  processes $(Y,Z)\in \mathbb{S}\times\mathcal{H}^d$  such that
\begin{align}\label{myq37}
\lim_{m\rightarrow \infty} \mathbf{E}\left[\sup\limits_{t\in[0,T]}
\big|{Y}^{(m)}_t-{Y}_t\big|^p+\bigg(\int^T_0\big|Z^{(m)}_t-Z_t\big|^2dt\bigg)^p\right]=0,\ \forall p\geq 1.
\end{align}
Set
\[
K_t=Y_t-Y_0+\int_0^tf(s,Y_s,Z_s)\, ds-\int_0^tZ_s\, dB_s.
\]
Using assumption (H2''), we obtain 
\begin{align*}
\lim_{m\rightarrow \infty}\mathbf{E}\left[\int^T_0 \left|f\left(t, Y_t^{(m-1)},\mathbf{P}_{Y^{(m-1)}_t}, Z_t^{(m)}\right)-f\left(t, Y_t,\mathbf{P}_{Y_t}, Z_t\right)\right|dt\right]=0,
\end{align*}
which implies that, as $m\to\infty$, \[\mathbf{E}\left[\sup\limits_{t\in[0,T]}\big|K_t-K_t^{(m)}\big|\right]\rightarrow 0.\]
In particular, we have $K_t=\lim\limits_{m\rightarrow\infty}K_t^{(m)}=\lim\limits_{m\rightarrow\infty}\mathbf{E}\left[K_t^{(m)}\right]=\mathbf{E}[K_t]$ and then $K$ is a deterministic, non-decreasing and continuous process. Finally, it follows from \eqref{my399} that
\[ \lim\limits_{m\rightarrow\infty}\mathbf{E}\left[\sup\limits_{t\in[0,T]}\left|\ell\left(t,Y^{(m)}_t\right)-\ell(t,Y_t)\right|\right]=0,\]
which indicates $\mathbf{E}[\ell(t,Y_t)]\geq 0$.
Moreover, recalling \cite[Lemma 13]{BH}, we have\[
\int_0^T \mathbf{E}[\ell(t,Y_t)]dK_t=\lim\limits_{m\rightarrow\infty} \int_0^T \mathbf{E}[\ell(t,Y^{(m)}_t)]dK^{(m)}_t=0,
\]
 which implies that $(Y,Z,K)\in\mathbb{S}\times\mathcal{H}^d\times\mathcal{A}$ is a deterministic solution to the quadratic mean reflected BSDE \eqref{my1}. The proof is complete.
\end{proof}

\appendix
\renewcommand\thesection{\normalsize Appendix}
\section{ }

\renewcommand\thesection{A}
\normalsize

This appendix is devoted to the proofs of Lemma \ref{myq7901} and Lemma \ref{myq7902}, which we give for the reader's convenience. The main idea is the same as in Lemma \ref{myq7900} or  \cite[Theorem 4.1]{HMW}.

\subsection{Proof of Lemma \ref{myq7901}}

In view of Lemma \ref{my1046} and Remark \ref{my3996}, we have for any $m\geq 1$,
\begin{align}\label{my100}
Y^{(m)}_t:= y^{(m)}_t+\sup\limits_{t\leq s\leq T}L_s\big(y^{(m)}_s\big), \ \ \forall t\in[0,T],
\end{align}
where $y^{(m)}_t$ is the solution to the following quadratic BSDE
 \begin{align}\label{myq1075}
y^{(m)}_t=\xi+\int^{T}_t f\left(s,Y^{(m-1)}_s,\mathbf{P}_{Y^{(m-1)}_s},z^{(m)}_s\right)ds-\int^{T}_t z^{(m)}_s dB_s.
\end{align}

\noindent Applying assertion (i) of Lemma \ref{my7} and \eqref{myq82} yields for any $t\in[0,T]$,
\small
\begin{align}\label{myq523}
	\begin{split}
	\exp\left\{{\gamma}\big|y^{(m)}_t\big|\right\}\leq \mathbf{E}_t\exp\left\{\gamma \left(|\xi|+\int^T_0 \alpha_sds+\beta(T-t)\bigg(\sup\limits_{s\in[t,T]}\big|Y^{(m-1)}_{s}\big|+\sup\limits_{s\in[t,T]}\mathbf{E}\left[\big|Y^{(m-1)}_{s}\big|\right]\bigg)\right)\right\}.
	\end{split}
\end{align}
\normalsize

\noindent Using {Doob's maximal inequality}, we get  for each  $m\geq 1, p\geq 2 $ and $t\in[0,T]$,
\begin{align}\label{myq1076}
	\begin{split}
	&\mathbf{E}\bigg[\exp\bigg\{{p\gamma}\sup\limits_{s\in[t,T]}\big|y^{(m)}_s\big|\bigg\}\bigg]\\
	& \hspace*{0.5cm}\leq 4 \mathbf{E}\left[\exp\bigg\{p\gamma \bigg(|\xi|+\int^T_0 \alpha_sds+\beta(T-t)\bigg(\sup\limits_{s\in[t,T]}|Y^{(m-1)}_{s}|+\sup\limits_{s\in[t,T]}\mathbf{E}\big[\big|Y^{(m-1)}_{s}\big|\big]\bigg)\bigg)\bigg\} \right].
	\end{split}
\end{align}
Recalling \eqref{my100}, we have  \[
\big|Y^{(m)}_t\big|\leq \big|y^{(m)}_t\big|+\sup\limits_{0\leq s\leq T}|L_s(0)|+\kappa\sup\limits_{t\leq s\leq T}\mathbf{E}\big[|y^{(m)}_s|\big].
\]
Set $\widetilde{\alpha}=\sup\limits_{0\leq s\leq T}|L_s(0)|+\int^T_0 \alpha_sds$. Using Jensen's inequality, we get for any $m\geq 1$, $p\geq 2 $ and $t\in[0,T]$,
{\small \begin{align*}	
	\begin{split}
	\mathbf{E}\left[\exp\bigg\{{p\gamma}\sup\limits_{s\in[t,T]}\big|Y^{(m)}_s\big|\bigg\}\right]&\leq e^{p\gamma\sup\limits_{0\leq s\leq T}|L_s(0)|}\,\mathbf{E}\left[\exp\big\{{p\gamma}\sup\limits_{s\in[t,T]}\big|y^{(m)}_s\big|\big\}\right]\mathbf{E}\left[\exp\big\{{\kappa p\gamma}\sup\limits_{s\in[t,T]}\big|y^{(m)}_s\big|\big\}\right]\\
	&\leq  e^{p\gamma\sup\limits_{0\leq s\leq T}|L_s(0)|}\,\mathbf{E}\left[\exp\big\{(2+2\kappa){p\gamma}\sup\limits_{s\in[t,T]}\big|y^{(m)}_s\big|\big\}\right]\\
	&\leq 4 \mathbf{E}\left[\exp\bigg\{(2+2\kappa)p\gamma \bigg(|\xi|+\widetilde{\alpha}+\beta(T-t)\sup\limits_{s\in[t,T]}|Y^{(m-1)}_{s}|\bigg)\bigg\}\right] \\
	&\hspace*{2cm} \times \mathbf{E}\left[\exp\bigg\{(2+2\kappa)p\gamma \beta(T-t)\sup\limits_{s\in[t,T]}|Y^{(m-1)}_{s}|\bigg\}	 \right].
	\end{split}
\end{align*}}

\noindent Choose a constant $h\in (0,T]$   depending only on $\beta$ and $\kappa$ such that
\begin{align}\label{myq7906}
(32+64\kappa)\beta h <1.\end{align}
In view of H\"{o}lder's inequality, we obtain that for any  $p\geq 2$,
 {\small \begin{align}\label{myq1071}
	\begin{split}
 &\mathbf{E}\left[\exp\bigg\{{p\gamma}\sup\limits_{s\in[T-h,T]}\big|Y^{(m)}_s\big|\bigg\}\right]
\\
&\leq 4 \mathbf{E}\left[\exp\bigg\{{(4+4\kappa) p\gamma}(|\xi|+\widetilde{\alpha})\bigg\}\right]^{\frac{1}{2}}\mathbf{E}\left[\exp\bigg\{(4+4\kappa)\beta hp\gamma\sup\limits_{s\in[T-h,T]}\big|Y^{(m-1)}_{s}\big|\bigg\}	\right]\\
&\leq 4\mathbf{E}\left[\exp\bigg\{{(8+8\kappa) p\gamma}|\xi|\bigg\}\right]^{\frac{1}{4}}\mathbf{E}\left[\exp\bigg\{{(8+8\kappa) p\gamma}\widetilde{\alpha}\bigg\}\right]^{\frac{1}{4}}  \mathbf{E}\left[\exp\bigg\{p\gamma\sup\limits_{s\in[T-h,T]}\big|Y^{(m-1)}_{s}\big|\bigg\}	\right]^{(4+4\kappa)\beta h}.
	\end{split}
\end{align}}

\noindent Define $ \rho = \frac{1}{1-(4+4\kappa)\beta h}$ and \begin{align*}
\mu:=
\begin{cases} \frac{T}{h}, \  &\text{if $\frac{T}{h}$ is an integer};\\
 [\frac{T}{h}]+1, \ &\text{otherwise}.
\end{cases}
\end{align*}
If $\mu=1$, it follows from  \eqref{myq1071} that for each $p\geq 2$ and $m\geq 1$,
\begin{align*}
\begin{split}
&\mathbf{E}\bigg[\exp\bigg\{p\gamma\sup\limits_{s\in[0,T]}\big|Y^{(m)}_s\big|\bigg\}
\bigg] \\
&\leq 4\mathbf{E}\bigg[\exp\bigg\{{(8+8\kappa) p\gamma}|\xi|\bigg\}\bigg]^{\frac{1}{4}}\mathbf{E}\bigg[\exp\bigg\{{(8+8\kappa) p\gamma}\widetilde{\alpha}\bigg\}\bigg]^{\frac{1}{4}} \mathbf{E}\bigg[\exp\bigg\{{p\gamma} \sup\limits_{s\in[0,T]}|Y^{(m-1)}_s|\bigg\}	 \bigg]^{(4+4\kappa)\beta h}.
\end{split}
\end{align*}
Iterating the above procedure $m$ times yields, given the definition of $\rho$,
 \begin{align}\label{myq698}
\begin{split}
\mathbf{E}\bigg[\exp\bigg\{{p\gamma}\sup\limits_{s\in[0,T]}\big|Y^{(m)}_s\big|\bigg\}
\bigg] \leq 4^{\rho}\mathbf{E}\bigg[\exp\bigg\{{(8+8\kappa) p\gamma}   |\xi|\bigg\}\bigg]^{\frac{\rho}{4}}\mathbf{E}\bigg[\exp\bigg\{{(8+8\kappa) p\gamma}  \widetilde{\alpha}\bigg\}\bigg]^{\frac{\rho}{4}},
\end{split}
\end{align}

\noindent which is uniformly bounded with respect to $m$ thanks to assumptions (H1'') and (H2''). If $\mu= 2$, proceeding identically as above, we have for any $p\geq 2$,
 \begin{align}\label{myq699}
 \begin{split}
&\mathbf{E}\bigg[\exp\bigg\{{p\gamma}\sup\limits_{s\in[T-h,T]}\big|Y^{(m)}_s\big|\bigg\}
\bigg]\leq 4^{\rho}\mathbf{E}\bigg[\exp\bigg\{{(8+8\kappa)p\gamma}  |\xi|\bigg\}\bigg]^{\frac{\rho}{4}}\mathbf{E}\bigg[\exp\bigg\{{(8+8\kappa) p\gamma}  \widetilde{\alpha}\bigg\}\bigg]^{\frac{\rho}{4}}.
\end{split}
\end{align}
We then consider the following quadratic reflected BSDE on time interval $[0,T-h]$:
\begin{align*}
\begin{cases}
&Y_t^{(m)}=Y_{T-h}^{(m)}+\int_t^{T-h} f(s,Y^{(m-1)}_s,\mathbf{P}_{Y^{(m-1)}_s}, Z^{(m)}_s)ds-\int_t^{T-h} Z^{(m)}_sdB_s+K^{(m)}_{{T-h}}-K^{(m)}_t,\\
& \mathbf{E}[\ell(t,Y^{(m)}_t)]\geq 0,\quad \forall t\in [0,T-h] \,\,\mbox{ and }\,\, \int_0^{T-h} \mathbf{E}[\ell(t,Y^{(m)}_t)] dK^{(m)}_t = 0.
\end{cases}
\end{align*}
According to the derivation of \eqref{myq698}, we deduce that
\begin{align*}
	\begin{split}
	&\mathbf{E}\bigg[\exp\bigg\{{p\gamma}\sup\limits_{s\in[0,T-h]}\big|Y^{(m)}_s\big|\bigg\}
	\bigg]\leq 4^{\rho}\mathbf{E}\bigg[\exp\bigg\{{(8+8\kappa) p\gamma}\big|Y_{T-h}^{(m)}\big|\bigg\}\bigg]^{\frac{\rho}{4}}\mathbf{E}\bigg[\exp\bigg\{{(8+8\kappa) p\gamma}  \widetilde{\alpha}\bigg\}\bigg]^{\frac{\rho}{4}}\\
	\hspace*{1cm}&\leq 4^{\rho+{\frac{\rho^2}{4}}}\mathbf{E}\bigg[\exp\bigg\{({8+8\kappa})^2 p\gamma|\xi|\bigg\}\bigg]^{\frac{\rho^2}{16}}\mathbf{E}\bigg[\exp\bigg\{({8+8\kappa})^2 p\gamma\widetilde{\alpha}\bigg\}\bigg]^{\frac{\rho^2}{16}}\mathbf{E}\bigg[\exp\bigg\{{(8+8\kappa) p\gamma} \widetilde{\alpha}\bigg\}\bigg]^{\frac{\rho}{4}},
	\end{split}
	\end{align*}
where we used \eqref{myq699} in the last inequality.
Putting the above inequalities together and applying H\"{o}lder's inequality again yields for any $p\geq 2$,
{\small \begin{align}\label{myq516}
	\begin{split}
	\mathbf{E}\bigg[\exp\bigg\{{p\gamma}\sup\limits_{s\in[0,T]}\big|Y^{(m)}_s\big|\bigg\}
	\bigg]
	&\leq \mathbf{E}\bigg[\exp\bigg\{{2p\gamma}\sup\limits_{s\in[0,T-h]}|Y^{(m)}_s|\bigg\}
	\bigg]^{\frac{1}{2}}	\mathbf{E}\bigg[\exp\bigg\{{2p\gamma}\sup\limits_{s\in[T-h,T]}|Y^{(m)}_s|\bigg\}
	\bigg]^{\frac{1}{2}}\\
	&\leq 4^{\rho+{\frac{\rho^2}{8}}}\mathbf{E}\bigg[\exp\bigg\{(8+8\kappa)^2 2p\gamma|\xi|\bigg\}\bigg]^{\frac{\rho}{8}+\frac{\rho^2}{32}}\mathbf{E}\bigg[\exp\bigg\{(8+8\kappa)^2 2p\gamma\widetilde{\alpha}\bigg\}\bigg]^{\frac{\rho}{4}+\frac{\rho^2}{32}},
	\end{split}
	\end{align}}which is also uniformly bounded with respect to $m$.

\noindent Iterating the above procedure $\mu$ times in the general case, we get
\begin{align*}
	\begin{split}
\sup\limits_{m\geq 0}\mathbf{E}\bigg[\exp\bigg\{p\gamma\sup\limits_{s\in[0,T]}|Y^{(m)}_s|\bigg\} \bigg]<\infty, \ \forall p\geq 1,
	\end{split}
\end{align*}
which together with \eqref{myq1076} implies that \begin{align}\label{myq10131}
	\begin{split}
\sup\limits_{m\geq 0}\mathbf{E}\bigg[\exp\bigg\{p\gamma\sup\limits_{s\in[0,T]}|y^{(m)}_s|\bigg\} \bigg]<\infty, \ \forall p\geq 1.
	\end{split}
\end{align}
It follows from Lemma \ref{my1046} and assumption (H4) that
\[
\sup\limits_{m\geq 0}K^{(m)}_T\leq \sup\limits_{0\leq s\leq T}|L_s(0)|+\sup\limits_{m\geq 0}\mathbf{E}\bigg[\sup\limits_{s\in[0,T]}\left|y^{(m)}_s\right|\bigg]<\infty.
\]
\noindent Finally, noting $Z^{(m)}=z^{(m)}$ and applying  \cite[Corollary 4]{BH2008} to the quadratic BSDE \eqref{myq1075} leads to
\begin{align*}
	\begin{split}
\sup\limits_{m\geq 0}\mathbf{E}\left[\bigg(\int^T_0\left|Z^{(m)}_t\right|^2dt\bigg)^p \right]<\infty, \ \forall p\geq 1,
	\end{split}
\end{align*}
which ends the proof.
\medskip

\subsection{Proof of Lemma \ref{myq7902}}

Without loss of generality, assume $f(t,y,v,\cdot)$ is concave, since the other case can be proved by a similar analysis, as discussed in Remark \ref{myrk11}.
For each fixed $m,q\geq 1$ and $\theta\in(0,1)$, we can define similarly $\delta_{\theta}\ell^{(m,q)}$, $\delta_{\theta}\widetilde{\ell}^{(m,q)}$ and $\delta_{\theta}\overline{\ell}^{(m,q)}$ for $y,z$.
Then, the pair of processes $(\delta_{\theta}y^{(m,q)},\delta_{\theta}z^{(m,q)})$
satisfies the following BSDE:
{\small \begin{align}\label{myq12}
\begin{split}
\delta_{\theta}y^{(m,q)}_t=&-\xi+\int^T_t\left(\delta_{\theta}f^{(m,q)}\left(s,\delta_{\theta}z^{(m,q)}_s\right)+\delta_{\theta}f_0^{(m,q)}(s)\right)ds-\int^T_t\delta_{\theta}z^{(m,q)}_sdB_s,
\end{split}
\end{align}}
where the  generator is given by
{\footnotesize \begin{align*}
& \delta_{\theta}f_0^{(m,q)}(t)\!=\!
\frac{1}{1-\theta}\left(f\left(t,Y^{(m+q-1)}_{t},\mathbf{P}_{Y^{(m+q-1)}_{t}}, z^{(m)}_t\right)\!-\!f\left(t,Y^{(m-1)}_{t},\mathbf{P}_{Y^{(m-1)}_{t}}, z^{(m)}_t\right)\right),\\
&
\delta_{\theta}f^{(m,q)}(t,z)\!=\!\frac{1}{1-\theta}\left(
\theta f\left(t,Y^{(m+q-1)}_{t},\mathbf{P}_{Y^{(m+q-1)}_{t}}, z^{(m+q)}_t\right) \! - \! f\left(t,Y^{(m+q-1)}_{t},\mathbf{P}_{Y^{(m+q-1)}_{t}}, -(1-\theta)z+\theta z^{(m+q)}_t\right)\right).
\end{align*}}

\noindent From assumption (H2''), we get
\begin{align*}
&\delta_{\theta}f_0^{(m,q)}(t)\leq \beta \left(|Y^{(m+q-1)}_{t}|+|\delta_{\theta}Y^{(m-1,q)}_t|+\mathbf{E}\bigg[\left|Y^{(m+q-1)}_{t}\right|+\left|\delta_{\theta}Y^{(m-1,q)}_t\right|\bigg]\right),\\
&\delta_{\theta}f^{(m,q)}(t,z)
 \leq -f\left(t,Y^{(m+q-1)}_t,\mathbf{P}_{Y^{(m+q-1)}_{t}}, -z\right)
\leq \alpha_t+\beta \left(\left|Y^{(m+q-1)}_t\right|+\mathbf{E}\bigg[\left|Y^{(m+q-1)}_{t}\right|\bigg]\right)+\frac{\gamma}{2}|z|^2.
\end{align*}
 Set  $ C_3:=2\sup\limits_{m}\mathbf{E}\big[\sup\limits_{s\in[0,T]}|Y^{(m)}_{s}|\big]<\infty$ (see Lemma \ref{myq7901})
and for any $m,q\geq 1$, denote
\begin{align*}
&\zeta^{(m,q)}=|\xi|+\beta T C_3+\int^T_0\alpha_sds+\beta T\left(\sup\limits_{s\in[0,T]}\left|Y^{(m-1)}_{s}\right|+\sup\limits_{s\in[0,T]}\left|Y^{(m+q-1)}_{s}\right|\right),\\
& \chi^{(m,q)}=2\beta T C_3+\int^T_0\alpha_sds+2 \beta T \left(\sup\limits_{s\in[0,T]}\left|Y^{(m+q-1)}_{s}\right|+\sup\limits_{s\in[0,T]}\left|Y^{(m-1)}_{s}\right|\right).\end{align*}

\noindent Applying assertion (ii) of Lemma \ref{my7} to  \eqref{myq12} yields for any $p\geq 1$,
{\small  \begin{align*}	
 \begin{split}
\exp\left\{{p\gamma}\big(\delta_{\theta}y^{(m,q)}_t\big)^+\right\} &\leq   \mathbf{E}_t\exp\bigg\{p\gamma \bigg(|\xi|+\chi^{(m,q)}+\beta (T-t)\bigg(\sup\limits_{s\in[t,T]}|\delta_{\theta}Y^{(m-1,q)}_{s}|+\sup\limits_{s\in[t,T]}\mathbf{E}[|\delta_{\theta}Y^{(m-1,q)}_{s}|]
 \bigg)\bigg)\bigg\}
 \end{split}
	\end{align*}
}

\noindent and in a similar way, we also have 
\begin{align*}	
 \begin{split}
 &\exp\left\{{p\gamma}\big(\delta_{\theta}\widetilde{y}^{(m,q)}_t\big)^+\right\}\\
 &\hspace*{1cm}\leq  \mathbf{E}_t\left[\exp\bigg\{p\gamma \bigg(|\xi|+\chi^{(m,q)}+\beta (T-t)\bigg(\sup\limits_{s\in[t,T]}|\delta_{\theta}\widetilde{Y}^{(m-1,q)}_{s}|
 +\sup\limits_{s\in[t,T]}\mathbf{E}[|\delta_{\theta}\widetilde{Y}^{(m-1,q)}_{s}|]
 \bigg)\bigg)\bigg)\bigg\}\right].
 \end{split}
\end{align*}

\noindent According to  the fact that
\begin{align*}
\big(\delta_{\theta}{y}^{(m,q)}\big)^-
\leq \big(\delta_{\theta}\widetilde{y}^{(m,q)}\big)^++2|y^{(m)}| \ \text{and}\ \big(\delta_{\theta}\widetilde{y}^{(m,q)}\big)^-
\leq \big(\delta_{\theta}{y}^{(m,q)}\big)^++2|y^{(m+q)}|,
\end{align*}
we derive, using H\"{o}lder's inequality and \eqref{myq523}, that
\begin{align*}
	\begin{split}
	&\exp\left\{p\gamma \big|\delta_{\theta}{y}^{(m,q)}_t\big|\right\}\vee \exp\left\{p\gamma \big|\delta_{\theta}\widetilde{y}^{(m,q)}_t\big|\right\}\\
	&\hspace*{1cm}\leq
	\exp\left\{{p\gamma}\bigg(\bigg( \delta_{\theta}{y}^{(m,q)}_t\bigg)^++\bigg(\delta_{\theta}\widetilde{y}^{(m,q)}_t\bigg)^++2\left|y^{(m)}_t\right|+2\left|y^{(m+q)}_t\right|\bigg)\right \}\\
	&\hspace*{1cm}\leq \mathbf{E}_t\left[\exp\bigg\{p\gamma \bigg(|\xi|+\chi^{(m,q)}+\beta (T-t)\bigg(\sup\limits_{s\in[t,T]}\delta_{\theta}\overline{Y}^{(m-1,q)}_{s}
    +\sup\limits_{s\in[t,T]}\mathbf{E}\left[\delta_{\theta}\overline{Y}^{(m-1,q)}_{s}\right]\bigg)\bigg)\bigg\}	\right]^2\\
    &\hspace{8cm} \times\exp\left\{{2p\gamma}\bigg(\left|y^{(m)}_t\right|+\left|y^{(m+q)}_t\right|\bigg)\right \} \\
    &\hspace*{1cm}\leq \mathbf{E}_t\bigg[\exp\bigg\{p\gamma \bigg(|\xi|+\chi^{(m,q)}+\beta (T-t)\bigg(\sup\limits_{s\in[t,T]}\delta_{\theta}\overline{Y}^{(m-1,q)}_{s}
    +\sup\limits_{s\in[t,T]}\mathbf{E}\left[\delta_{\theta}\overline{Y}^{(m-1,q)}_{s}\right]\bigg)\bigg)\bigg\}\bigg]^2\\
    & \hspace*{8cm} \times \mathbf{E}_t\left[\exp\bigg\{4p\gamma\zeta^{(m,q)} \bigg\}	\right].
	\end{split}
\end{align*}

\noindent In view of {Doob's maximal inequality} and
H\"{o}lder's inequality, we obtain that for all $p> 1$ and $t\in[0,T]$
\begin{align*}
	\begin{split}
	&\mathbf{{E}}\left[\exp\bigg\{p\gamma \sup\limits_{s\in[t,T]}\delta_{\theta}\overline{y}^{(m,q)}_s\bigg\}\right]
	\\ &\leq 4 \mathbf{E}\left[\exp\bigg\{8p\gamma \bigg(|\xi|+{\chi}^{(m,q)} +\beta (T-t)\bigg(\sup\limits_{s\in[t,T]}\delta_{\theta}\overline{Y}^{(m-1,q)}_{s}
    +\sup\limits_{s\in[t,T]}\mathbf{E}\left[\delta_{\theta}\overline{Y}^{(m-1,q)}_{s}\right]\bigg) \bigg)\bigg\}\right]^{\frac{1}{2}} \\
    &\hspace*{3cm} \times \mathbf{E}\left[\exp\bigg\{16p\gamma \zeta^{(m,q)} \bigg\} \right]^{\frac{1}{2}}\\ 
    &\leq  4  \mathbf{E}\left[\exp\bigg\{8p\gamma \bigg(|\xi|+{\chi}^{(m,q)} +\beta (T-t)\sup\limits_{s\in[t,T]}\delta_{\theta}\overline{Y}^{(m-1,q)}_{s}
    \bigg)\bigg\}	\right]^{\frac{1}{2}} \\
    &\hspace*{3cm} \times\mathbf{E}\left[\exp\bigg\{8\beta(T-t) p\gamma  \sup\limits_{s\in[t,T]}\delta_{\theta}\overline{Y}^{(m-1,q)}_{s} \bigg)\bigg\}	\right]^{\frac{1}{2}}  \mathbf{E}\left[\exp\bigg\{16p\gamma \zeta^{(m,q)} \bigg\}	\right]^{\frac{1}{2}}.
\end{split}
\end{align*}

\noindent Set $ C_4:=\sup\limits_{0\leq s\leq T}|L_s(0)|+2\kappa\sup\limits_{m}\mathbf{E}\big[\sup\limits_{s\in[0,T]}|y^{(m)}_{s}|\big]<\infty$ (see \eqref{myq10131}). 
Recalling \eqref{my100} and assumption (H4),
\begin{align*}
\delta_{\theta}\overline{Y}^{(m,q)}_t\leq \delta_{\theta}\overline{y}^{(m,q)}_t+2\kappa\sup\limits_{t\leq s\leq T}\mathbf{E}\left[\delta_{\theta}\overline{y}^{(m,q)}_t\right]+2C_4,
\end{align*}
which together with  Jensen's inequality implies that for each  $p\geq 1$ and $t\in[0,T]$
\begin{align*}
	\begin{split}
        &\mathbf{E}\left[\exp\bigg\{p\gamma \sup\limits_{s\in[t,T]}\delta_{\theta}\overline{Y}^{(m,q)}_s\bigg\}\right]\leq  e^{2p\gamma C_4}\mathbf{E}\left[\exp\bigg\{(2+4\kappa)p\gamma \sup\limits_{s\in[t,T]}\delta_{\theta}\overline{y}^{(m,q)}_s\bigg\}\right]\\ &\leq  4  \mathbf{E}\left[\exp\bigg\{(16+32\kappa)p\gamma \bigg(|\xi|+{\chi}^{(m,q)}+C_4 +\beta (T-t)\sup\limits_{s\in[t,T]}\delta_{\theta}\overline{Y}^{(m-1,q)}_{s}
        \bigg)\bigg\}	\right]^{\frac{1}{2}} \\
        &\hspace*{1cm} \times\mathbf{E}\left[\exp\bigg\{(16+32\kappa)\beta(T-t) p\gamma  \sup\limits_{s\in[t,T]}\delta_{\theta}\overline{Y}^{(m-1,q)}_{s} \bigg)\bigg\}	\right]^{\frac{1}{2}}  \mathbf{E}\left[\exp\bigg\{(32+64\kappa)p\gamma \zeta^{(m,q)} \bigg\}	\right]^{\frac{1}{2}}.
	\end{split}
\end{align*}

\noindent Choosing $h$ as in \eqref{myq7906}, we have
{ \small \begin{align}	\label{myq512}
	\begin{split}
	&\mathbf{{E}}\left[\exp\bigg\{p\gamma \sup\limits_{s\in[T-h,T]}\delta_{\theta}\overline{Y}^{(m,q)}_s\bigg\}\right]
	\\ &\leq 4\mathbf{{E}}\left[\exp\bigg\{{(64+128\kappa) p\gamma}|\xi|\bigg\}		\right]^{\frac{1}{8}}\mathbf{{E}}\left[\exp\bigg\{(64+128\kappa) p\gamma\big({\chi}^{(m,q)}+C_4\big)\bigg\}		\right]^{\frac{1}{8}}\\ & \ \ \ \ \ \ \ \ \ \  \times \mathbf{E}\left[\exp\bigg\{(32+64\kappa) p\gamma \zeta^{(m,q)} \bigg\}	 \right]^{\frac{1}{2}}
	\mathbf{{E}}\left[\exp\bigg\{p\gamma\sup\limits_{s\in[T-h,T]}\delta_{\theta}\overline{Y}^{(m-1,q)}_{s}\bigg\}		\right]^{(16+32\kappa)\beta h}.
	\end{split}
	\end{align}
}

\noindent Set $ \widetilde{\rho} = \frac{1}{1-(16+32\kappa)\beta h}$. If $\mu=1$, it follows from \eqref{myq512} that for each $p\geq 1$ and $m,q\geq 1$,
\begin{align*}
\begin{split}
    &\mathbf{E}\left[\exp\bigg\{p\gamma\sup\limits_{s\in[0,T]}\delta_{\theta}\overline{Y}^{(m,q)}_s\bigg\}\right]\\
    &\leq 4^{\textcolor{black}{\widetilde{\rho}}}\mathbf{{E}}\left[\exp\bigg\{(64+128\kappa) p\gamma|\xi|\bigg\}		\right]^{\frac{\widetilde{\rho}}{8}}\sup\limits_{m,q\geq 1}\mathbf{{E}}\left[\exp\bigg\{(64+128\kappa) p\gamma\big({\chi}^{(m,q)}+C_4\big)\bigg\}		\right]^{\frac{\widetilde{\rho}}{8}}\\
	&\hspace*{1cm}  \times\sup\limits_{m,q\geq 1}\mathbf{E}\left[\exp\bigg\{(32+64\kappa)p\gamma \zeta^{(m,q)} \bigg\}	\right]^{\frac{\widetilde{\rho}}{2}} \mathbf{{E}}\left[\exp\bigg\{p\gamma\sup\limits_{s\in[0,T]}\delta_{\theta}\overline{Y}^{(1,q)}_{s}\bigg\}		\right]^{(16\beta h+32\kappa\beta h)^{m-1}}.
\end{split}
\end{align*}

\noindent The result from Lemma \ref{myq7901} insures that for any $\theta\in(0,1)$
\[
\lim_{m\rightarrow \infty}\sup\limits_{q\geq 1}\mathbf{{E}}\left[\exp\bigg\{p\gamma\sup\limits_{s\in[0,T]}\delta_{\theta}\overline{Y}^{(1,q)}_{s}\bigg\}		\right]^{(16\beta h+32\kappa\beta h)^{m-1}}=1,
\]
which implies that
\begin{align*}
    &\sup\limits_{\theta\in(0,1)}\lim_{m\rightarrow \infty}\sup\limits_{q\geq 1}\mathbf{E}\left[\exp\big\{p\gamma\sup\limits_{s\in[0,T]}\delta_{\theta}\overline{Y}^{(m,q)}_s\big\}\right] \\
    &\hspace*{1cm}\leq  4^{\widetilde{\rho}}\mathbf{{E}}\left[\exp\bigg\{(64+128\kappa) p\gamma|\xi|\bigg\}		\right]^{\frac{\widetilde{\rho}}{8}}\sup\limits_{m,q\geq 1}\mathbf{{E}}\left[\exp\bigg\{(64+128\kappa) p\gamma\big({\chi}^{(m,q)}+C_4\big)\bigg\}\right]^{\frac{\widetilde{\rho}}{8}}\\
	&\hspace*{1cm} \times\sup\limits_{m,q\geq 1}\mathbf{E}\left[\exp\bigg\{(32+64\kappa)p\gamma \zeta^{(m,q)} \bigg\}	\right]^{\frac{\widetilde{\rho}}{2}}<\infty.
\end{align*}
If $\mu=2$, in view of the derivation of \eqref{myq516}, we conclude that for any $p\geq 1$,
 \begin{align*}
	\begin{split}
	&\mathbf{E}\left[\exp\big\{{p\gamma}\sup\limits_{s\in[0,T]}\delta_{\theta}\overline{Y}^{(m,q)}_s\big\}
	\right]\\
	&\leq 4^{\widetilde{\rho}+{\frac{\widetilde{\rho}^2}{16}}}\mathbf{E}\left[\exp\bigg\{(64+128\kappa)^2 2p\gamma|\xi|\bigg\}		\right]^{{\frac{\widetilde{\rho}}{16}+\frac{\widetilde{\rho}^2}{128}}}\sup\limits_{m,q\geq 1}\mathbf{{E}}\left[\exp\bigg\{(64+128\kappa)^2 2p\gamma\big({\chi}^{(m,q)}+C_4\big)\bigg\}		\right]^{{\frac{\widetilde{\rho}}{8}+\frac{\widetilde{\rho}^2}{128}}}\\
	&\hspace*{1cm} \times\sup\limits_{m,q\geq 1}\mathbf{E}\left[\exp\bigg\{(32+64\kappa)(64+128\kappa)2p\gamma \zeta^{(m,q)} \bigg\}	\right]^{{\frac{\widetilde{\rho}}{2}+\frac{\widetilde{\rho}^2}{32}}}\\
	&\hspace*{1cm}  \times \mathbf{{E}}\left[\exp\bigg\{(64+128\kappa)2p\gamma\sup\limits_{s\in[0,T]}\delta_{\theta}\overline{Y}^{(1,q)}_{s}\bigg\}		\right]^{({\frac{1}{2}+\frac{\widetilde{\rho}}{16})}(16\beta h+32\kappa\beta h)^{m-1}},
	\end{split}
	\end{align*}
which also implies the desired assertion when $\mu=2$. Iterating the above procedure $\mu$ times in the general case, we complete the proof.

\end{document}